\documentclass[12pt]{amsart}
\usepackage[margin=0.5in]{geometry}
\newcommand{\R}{\mathbb{R}}

\renewcommand{\epsilon}{\varepsilon}

\usepackage[english]{babel}
\usepackage[alphabetic]{amsrefs}
\usepackage{amsmath,amssymb,amsfonts,amsthm,enumerate}
\usepackage{url}
\usepackage{graphicx,epstopdf,color}

\numberwithin{equation}{section}

\newtheorem{theorem}{Theorem}[section]
\newtheorem{lemma}[theorem]{Lemma}
\newtheorem{definition}[theorem]{Definition}

\newtheorem{corollary}[theorem]{Corollary}

\renewcommand{\le}{\leqslant}

\renewcommand{\ge}{\geqslant}

\newcommand{\Per}{\operatorname{Per}}
\renewcommand{\epsilon}{\varepsilon}
\newcommand{\e}{\varepsilon}

\allowdisplaybreaks

\title[Boundary continuity of nonlocal minimal surfaces]{Boundary continuity of nonlocal minimal surfaces \\ in domains with singularities \\ and a problem posed by
Borthagaray, Li, and Nochetto}

\thanks{OS was supported by the NSF grant DMS-2055617.
EV was supported by the Australian Laureate Fellowship FL190100081
and by the UWA HiCi Fund.}

\author[Serena Dipierro,
Ovidiu Savin, and Enrico Valdinoci]{Serena Dipierro${}^{(1)}$\and
Ovidiu Savin${}^{(2)}$
\and
Enrico Valdinoci${}^{(1)}$}

\begin{document}
\maketitle

{\scriptsize \begin{center} (1) -- Department of Mathematics and Statistics\\
University of Western Australia\\ 35 Stirling Highway, WA6009 Crawley (Australia)\\
\end{center}
\scriptsize \begin{center} (2) --
Department of Mathematics\\
Columbia University\\
2990 Broadway, NY 10027
New York (USA)
\end{center}
\bigskip

\begin{center}
E-mail addresses:
{\tt serena.dipierro@uwa.edu.au},
{\tt savin@math.columbia.edu},
{\tt enrico.valdinoci@uwa.edu.au}
\end{center}
}
\bigskip\bigskip

\begin{abstract}
Differently from their classical counterpart, nonlocal minimal surfaces are known to present boundary
discontinuities, by sticking at the boundary of smooth domains.

It has been observed numerically by
J. P. Borthagaray, W. Li, and R. H. Nochetto ``that stickiness is larger near the concave portions of the boundary than near the
convex ones, and that it is absent in the corners of the square'', leading to the conjecture
``that there is a relation between the amount of stickiness on~$\partial\Omega$
and the nonlocal mean curvature of~$\partial\Omega$''.

In this paper, we give a positive answer to this conjecture, by showing that the nonlocal minimal surfaces are continuous at convex corners of the domain boundary and discontinuous at concave corners.

More generally, we show that boundary continuity for nonlocal minimal surfaces holds true at all points in which the domain is not better than~$C^{1,s}$, with the singularity pointing outward,
while, as pointed out by a concrete example, discontinuities may occur at all point in which the domain possesses an interior touching set of class~$C^{1,\alpha}$ with~$\alpha>s$.
\end{abstract}

\section{Introduction}

\subsection{Motivations}
While classical minimal surfaces arise as minimizers of the perimeter functional
and model classical surface tensions, nonlocal minimal surfaces aim at capturing long-range interactions induced by kernels with\footnote{Integrable kernels have been also taken into account
and produce a different theory, see~\cite{MR3930619}.} a fat tail. The systematic study of nonlocal minimal surfaces
started in~\cite{MR2675483} and covered many topics, such as interior regularity~\cite{MR3090533, MR3107529, MR3680376, MR3798717, MR3981295, MR4116635},
geometric flows~\cite{MR2487027, MR3401008, MR3713894, MR3778164, MR3951024, MR4000255, MR4104832, MR4175821}, front propagation~\cite{MR2564467},
nonlocal isoperimetric inequalities~\cite{MR2469027, MR2799577, MR3264796, MR3322379, MR3412379}, surfaces of constant nonlocal mean curvature~\cite{MR3485130, MR3744919, MR3770173, MR3836150, MR3881478}, capillarity theories~\cite{MR3717439, MR4404780}, limit embeddings~\cite{MR1945278, MR2033060},
long-range phase transitions~\cite{MR1372427, MR2948285},
fractal analysis~\cite{MR1111612, MR3912427}, problems with higher codimension~\cite{MR3733825, MR4058510},
just to name a few directions.

An interesting feature discovered in~\cite{MR3596708} and further analyzed in~\cite{MR3926519, MR4104542, MR4178752, MR4184583, MR4392355, MR4548844}
consists in a boundary behavior for nonlocal minimal surfaces which is significantly different from the classical case. Namely, at least in convex domains, classical minimal surfaces detach from the boundary in a transversal way. Conversely, nonlocal minimal surfaces can adhere to the boundary of the domain (and actually present the strong tendency to do so). This phenomenon, which is also related to an obstacle problem for nonlocal minimal surfaces~\cite{MR3532394}, affects the boundary regularity, since, in the presence of stickiness, nonlocal minimal surfaces do not attain their external datum in a continuous way.

In this regard, an intriguing conjecture was posed by
J. P. Borthagaray, W. Li, and R. H. Nochetto~\cite{MR4294645}, according to which
stickiness never occurs at convex corners of the boundary, while typically manifesting itself at concave corners.

This article is motivated by this conjecture, which we aim to address in our main results.

\subsection{Main results}
{F}rom now on, we suppose that~$n\ge2$ and~$\Omega$ will denote an open subset of~$\R^{n+1}$ with Lipschitz boundary.

Given~$E\subseteq\R^{n+1}$, we consider the $s$-perimeter functional in~$\Omega$ defined by
$$ \Per_s(E,\Omega):= L_s(E\cap\Omega,E^c\cap\Omega)+L_s(E\cap\Omega,E^c\cap\Omega^c)+L_s(E\cap\Omega^c,E^c\cap\Omega),$$
where~$s\in(0,1)$ and
$$ L_s(A,B):=\iint_{A\times B} \frac{dx\,dy}{|x-y|^{n+1+s}}.$$
As usual, the superscript~``$c$'' denotes the complementary set in~$\R^{n+1}$,
and all sets are implicitly assumed to be measurable.

A topical objective of interest is the family of minimizers for the $s$-perimeter, as recalled here below:

\begin{definition}\label{SL-D}
If~$\Omega$ is bounded, we say that~$E\subseteq\R^{n+1}$ is an~$s$-minimal set in~$\Omega$ if~$\Per_s(E,\Omega)<+\infty$ and
$$ \Per_s(E,\Omega)\le\Per_s(F,\Omega)$$
for all sets~$F\subseteq\R^{n+1}$ such that~$E\setminus\Omega=F\setminus\Omega$.
\end{definition}

In this setting, \cite[Theorem 3.2]{MR2675483} ensures, given~$E_0\subseteq\R^{n+1}$,
the existence of an~$s$-minimal set~$E$ in~$\Omega$ such that~$E\setminus\Omega=E_0\setminus\Omega$.

We will now focus our attention on the case of cylindrical domains, i.e. we assume from now on that~$\Omega=\omega\times\R$, for some bounded set~$\omega$ in~$\R^{n}$. In this framework,
since~$\Omega$ is unbounded, Definition~\ref{SL-D} needs to be slightly modified (see also~\cite{MR3827804} for additional information\footnote{As a technical observation, we mention that a natural class of minimizers in the case of graphical external data is given by
that of~$s$-minimal graphs, namely of $s$-minimal sets which can be written as graphs, say, in the $(n+1)$th coordinate direction. We refer to~\cite[Theorem~1.2]{MR3516886}
and, more generally,~\cite[Theorem~1.3]{MR4279395} for existence results of $s$-minimal graphs.
See also~\cite{MR3934589} for a specific regularity theory for $s$-minimal graphs.
However, the setting of $s$-minimal graphs will not be explicitly used in this paper, to maintain the exposition as simple as possible.}
on situations of this type):

\begin{definition}
We say that~$E\subseteq\R^{n+1}$ is an~$s$-minimal set in~$\Omega$ if
it is an $s$-minimal set in every bounded open set~$\Omega'$ contained in~$\Omega$
according to Definition~\ref{SL-D}.
\end{definition}

Since we deal with cylindrical domains~$\Omega=\omega\times\R\subseteq\R^{n+1}$, it is often useful to denote points in~$\R^{n+1}$ by~$X=(x,x_{n+1})=(x',x_{n},x_{n+1})\in\R^{n-1}\times\R\times\R$.
The $n$-dimensional ball centered at~$p\in\R^n\times\{0\}$ and of radius~$\rho>0$ will be denoted by
$$ B_\rho(p):=\Big\{ (x,0)\in\R^{n}\times\{0\} {\mbox{ s.t. }} |x-p|<\rho
\Big\}.$$
We also consider the corresponding cylinder in~$\R^{n+1}$ given by
$$ C_\rho(p):=B_\rho(p)\times\R.$$
When~$p=0$, we use the short notations~$B_\rho$ and~$C _\rho$.

For an $(n+1)$-dimensional ball, we use the notation, given~$P\in\R^{n+1}$,
$$ {\mathcal{B}}_\rho(P):=
\Big\{ X\in\R^{n+1} {\mbox{ s.t. }} \big|X-P\big|<\rho
\Big\}.$$

\begin{figure}[h]
\fbox{\includegraphics[height=4.5cm]{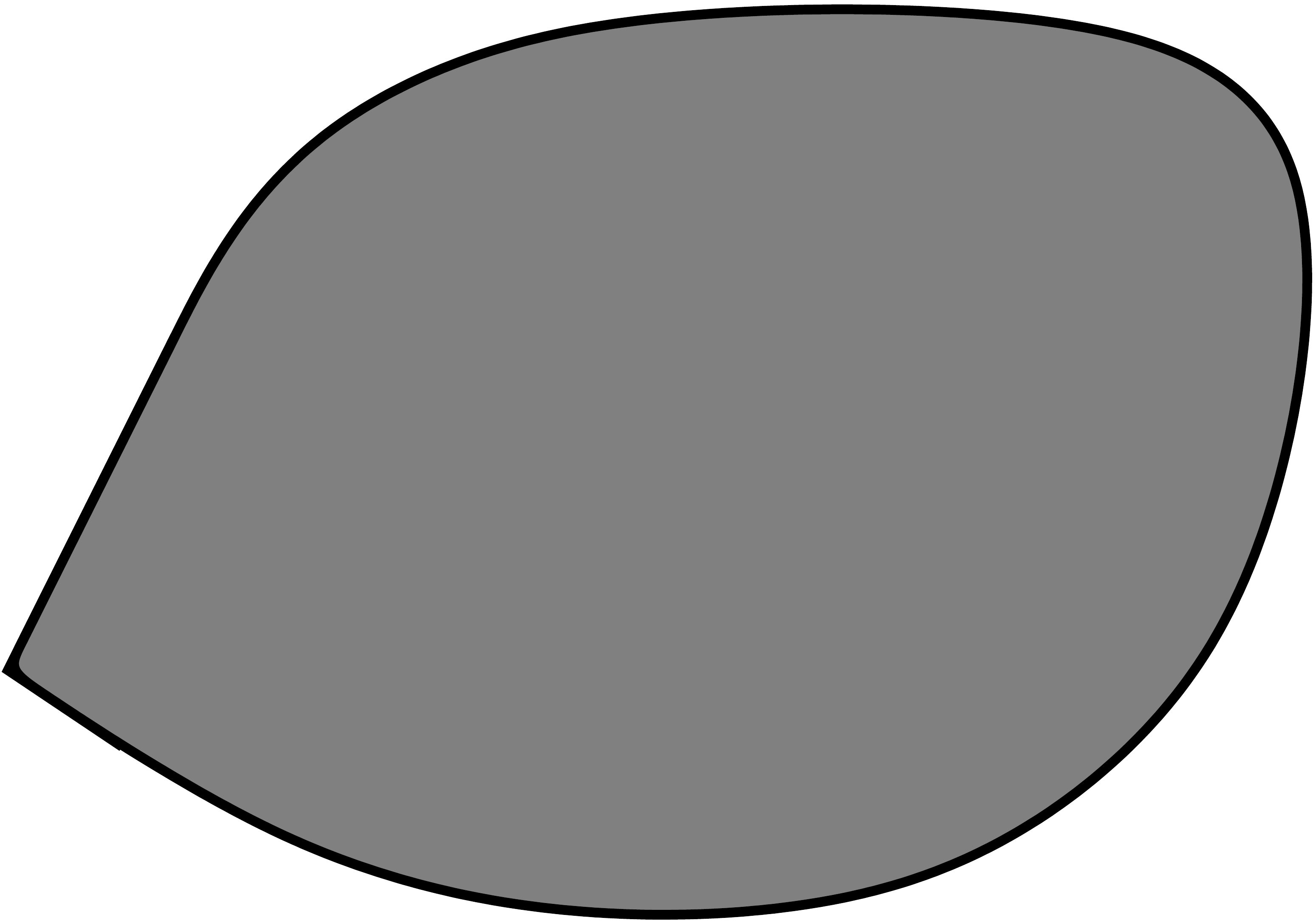}}$\,$
\fbox{\includegraphics[height=4.5cm]{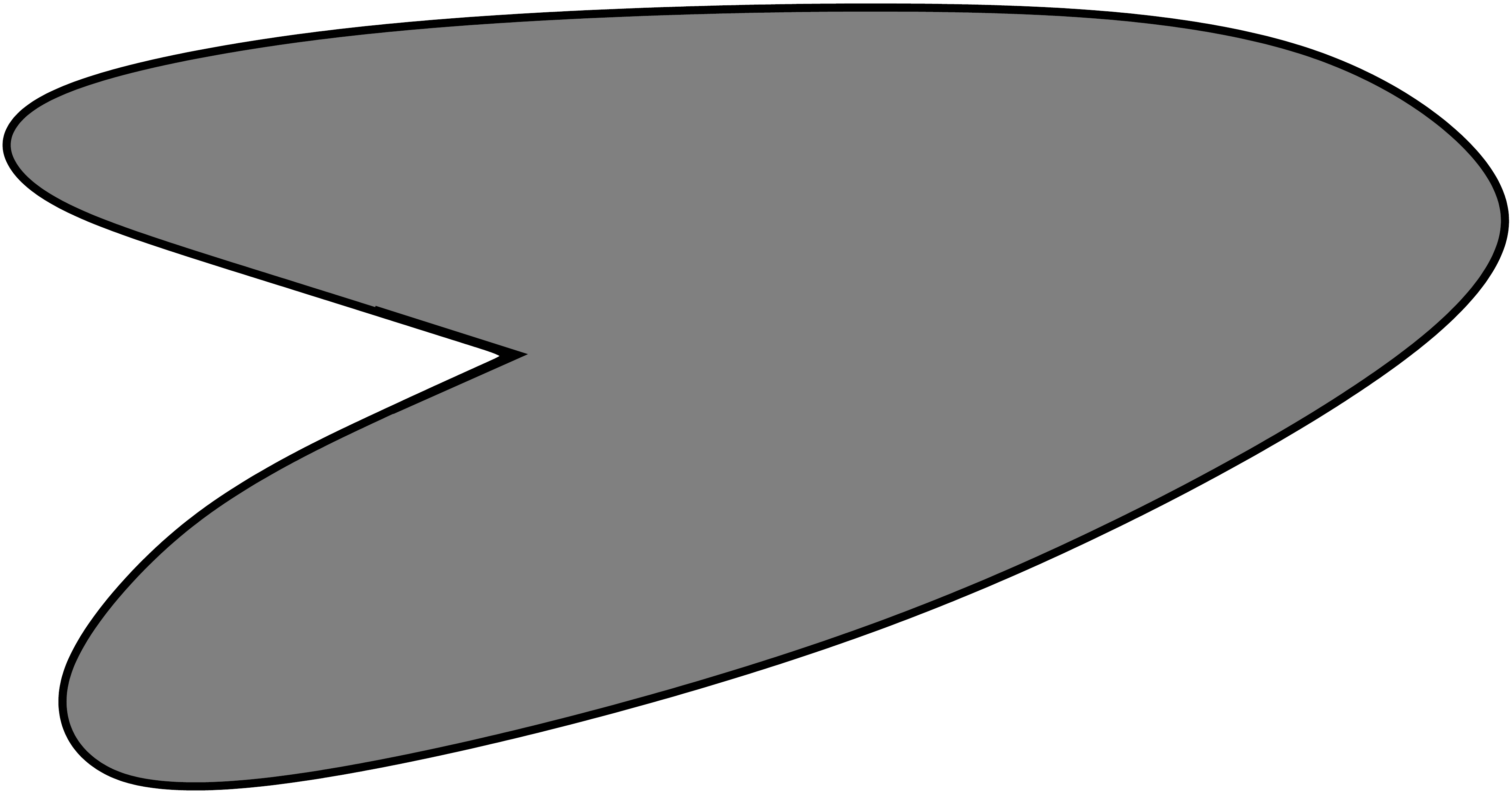}}
\caption{Left: a domain with an outward pointing singularity. Right:
a domain with an inward pointing singularity.}\label{FIGAGG}
\end{figure}

Our main result establishes continuity of the $s$-minimal sets at the points in which
the domain is not better than~$C^{1,s}$, with the singularity pointing outward of the domain (see Figure~\ref{FIGAGG} for a diagram of domains
with singularities pointing outward and inward).
The precise statement goes as follows:

\begin{theorem}[Continuity of the $s$-minimal sets for domains with outward singularities]\label{THM:CON}
Suppose that there exist~$\rho>0$ and~$\varphi:\R^{n-1}\to[0,+\infty)$, with
\begin{equation}\label{KASMFMA}
\varphi(0)=0\end{equation}
and
\begin{equation}\label{CBET}
\varphi(x')\ge c|x'|^\beta\qquad{\mbox{for all}}\quad x'\in \R^{n-1} \quad{\mbox{ with }} |x'|<\rho 
,\end{equation}
for some~$c>0$ and~$\beta\in(0,s+1]$,
and such that
\begin{equation}\label{INC:1} \omega\cap B_\rho =\{ x_{n}>\varphi(x')\}\cap B_\rho.\end{equation}
Assume that
\begin{equation}\label{INC:22} E_0=\{x_{n+1}<\psi(x',x_n)\}\end{equation}
for some~$\psi\in L^\infty_{\rm loc}(\R^n)$ such that
\begin{equation}\label{INC:2}
\psi(x',x_n)\le0 {\mbox{ for all }}(x',x_n)\in B_\rho.
\end{equation}

Let~$E$ be an $s$-minimal set in~$\Omega$ with~$E\setminus\Omega=E_0\setminus\Omega$.

Then, for every~$\e>0$ there exists~$\delta>0$ such that
\begin{equation}\label{CONT:1}
E\cap C_\delta\subseteq\{ x_{n+1}\le\e\}.
\end{equation}
\end{theorem}

Some comments about Theorem~\ref{THM:CON} are in order.
Firstly,  conditions~\eqref{KASMFMA}, \eqref{CBET}, and~\eqref{INC:1} describe the geometry of the $n$-dimensional domain~$\omega$ (and therefore of the $(n+1)$-dimensional cylinder~$\Omega$): in 
a nutshell, these assumptions state that the origin belongs to the boundary of~$\omega$, that the domain is not better than~$C^{1,s}$ in the vicinity of the origin, and that the singularity points
``outward''.

Also, conditions~\eqref{INC:22} and~\eqref{INC:2} deal with the external datum
and basically say that this datum is below~$\{x_{n+1}=0\}$ in the vicinity of the origin.

The thesis obtained in~\eqref{CONT:1} thus controls the oscillations of the $s$-minimal set near the origin.

Obviously, up to reverting the vertical direction, the inequality signs in~\eqref{INC:22},
\eqref{INC:2} and~\eqref{CONT:1} can be reverted (with~$\e$ replaced by~$-\e$ in~\eqref{CONT:1}). Consequently, if~\eqref{INC:22}
and~\eqref{INC:2} are replaced by
\begin{equation}\begin{split}\label{BSODIKnC1}&E_0=\{x_{n+1}=\psi(x',x_n)\}\\ {\mbox{with}}\qquad&
\psi(x',x_n)=0 {\mbox{ for all }}(x',x_n)\in B_\rho,
\end{split}\end{equation}
then the thesis in~\eqref{CONT:1} can be strengthen into \begin{equation}\label{BSODIKnC2}
E\cap C_\delta\subseteq\{ |x_{n+1}|\le\e\},\end{equation} which can be seen as a continuity result.

In this spirit, we stress that boundary continuity for $s$-minimal sets is somewhat a ``rare'' phenomena
and typically jump discontinuities have to be expected, as established in~\cite{MR3596708, MR4104542, MR4178752, MR4392355, MR4548844}; see also~\cite{MR3982031, MR4294645} for several accurate numerical simulations that
showcase such discontinuities in this setting. Therefore, the continuity result provided by
Theorem~\ref{THM:CON} can be seen as an interesting counterpart of the more common boundary discontinuity: roughly speaking, this continuity is obtained thanks to domains which are ``not regular
enough'', with a direction of singularity making the long-range effects coming from the external data
by some means negligible with respect to the localized interaction reminding surface tension
(but of course some care is needed to make such a statement precise and quantitatively coherent).\medskip

As a byproduct of Theorem~\ref{THM:CON}, we can give a positive answer to the thought-provoking conjecture posed by J. P. Borthagaray, W. Li, and R. H. Nochetto (see~\cite[page~25]{MR4294645}),
who observed from numerical simulations in three-dimensional cylinders that jumpt discontinuity is ``absent at the convex corners of~$\Omega$''. We prove this conjecture as a direct consequence of Theorem~\ref{THM:CON}:

\begin{corollary}[Borthagaray-Li-Nochetto Conjecture]\label{THM:CON-cor}
Let~$\omega$ be a two-dimensional domain with a convex corner at the origin.

Let~$E_0$ be a smooth graph vanishing in a neighborhood of the origin
and let~$E$ be an $s$-minimal set in~$\Omega$ with~$E\setminus\Omega=E_0\setminus\Omega$.

Then, $E$ is continuous at the origin, in the sense that
for every~$\e>0$ there exists~$\delta>0$ such that
\begin{equation*}
E\cap C_\delta\subseteq\{ |x_{3}|\le\e\}.
\end{equation*}
\end{corollary}

The result in Corollary~\ref{THM:CON-cor} is showcased\footnote{The figures of this paper
are just qualitative sketches, not numerical simulations, and have to be taken with a pinch of salt.} in Figure~\ref{FIG1}.

\begin{figure}[h]
\fbox{\includegraphics[height=5cm]{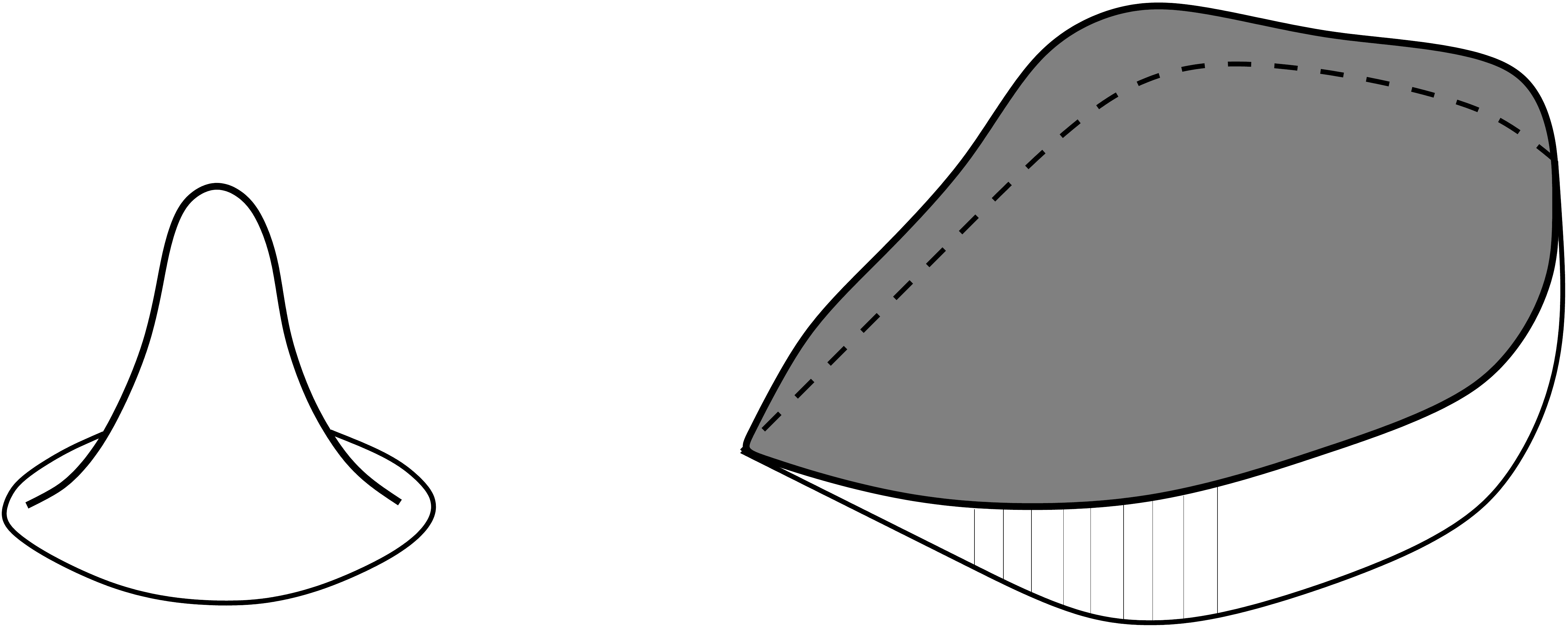}}
\caption{Boundary continuity for ``outward pointing'' corners.}\label{FIG1}
\end{figure}

We stress that the convexity assumption in Corollary~\ref{THM:CON-cor} cannot be removed. As mentioned in~\cite[page~25]{MR4294645}, numerical evidence suggests ``that stickiness is most pronounced at the reentrant corner''. The discontinuity of $s$-minimal sets at these concave corners is indeed a rather general phenomenon, and actually occurs even in the absence of corners, namely for smooth sets (and actually it suffices for~$\omega$ to possess an inner touching condition at the origin of class~$C^{1,\alpha}$ with~$\alpha>s$, which is obviously the case for concave angles
and which also shows the optimality of the exponent~$\beta$ in~\eqref{CBET}).
The precise result that we propose in this framework goes as follows:

\begin{theorem}[Boundary discontinuity for ``inward pointing'' domains]\label{AKgSMX:CONCA}
Assume that~$0\in \partial\omega$ and that there exists a bounded, $n$-dimensional set~$S$
with boundary of class~$C^{1,\alpha}$, for some~$\alpha\in(s,1)$, contained in~$\omega$ and such that~$0\in\partial S$.

Then, for every~$\epsilon>0$, there exist~$\delta>0$, $\psi\in C^\infty_0(\R^n\setminus\overline{\omega},[0,\e])$ and an $s$-minimal set~$E$ in~$\Omega$
such that
$$E\setminus\Omega=E_0\setminus\Omega,$$
with
\begin{equation*} E_0:=\{x_{n+1}<\psi(x',x_n)\},\end{equation*}
and
\begin{equation}\label{siCONT:1}
E\cap(S\times \R)\,\supseteq\,  (S\cap B_\delta)\times[0,\delta].
\end{equation}
\end{theorem}

Notice that~\eqref{siCONT:1} gives that the $s$-minimal set is discontinuous at the origin, presenting
a jump of at least~$\delta$. Interestingly, this discontinuity can be produced by an arbitrarily small perturbation (as encoded by the parameter~$\e$ in Theorem~\ref{AKgSMX:CONCA}).

\begin{figure}[h]
\fbox{\includegraphics[height=5cm]{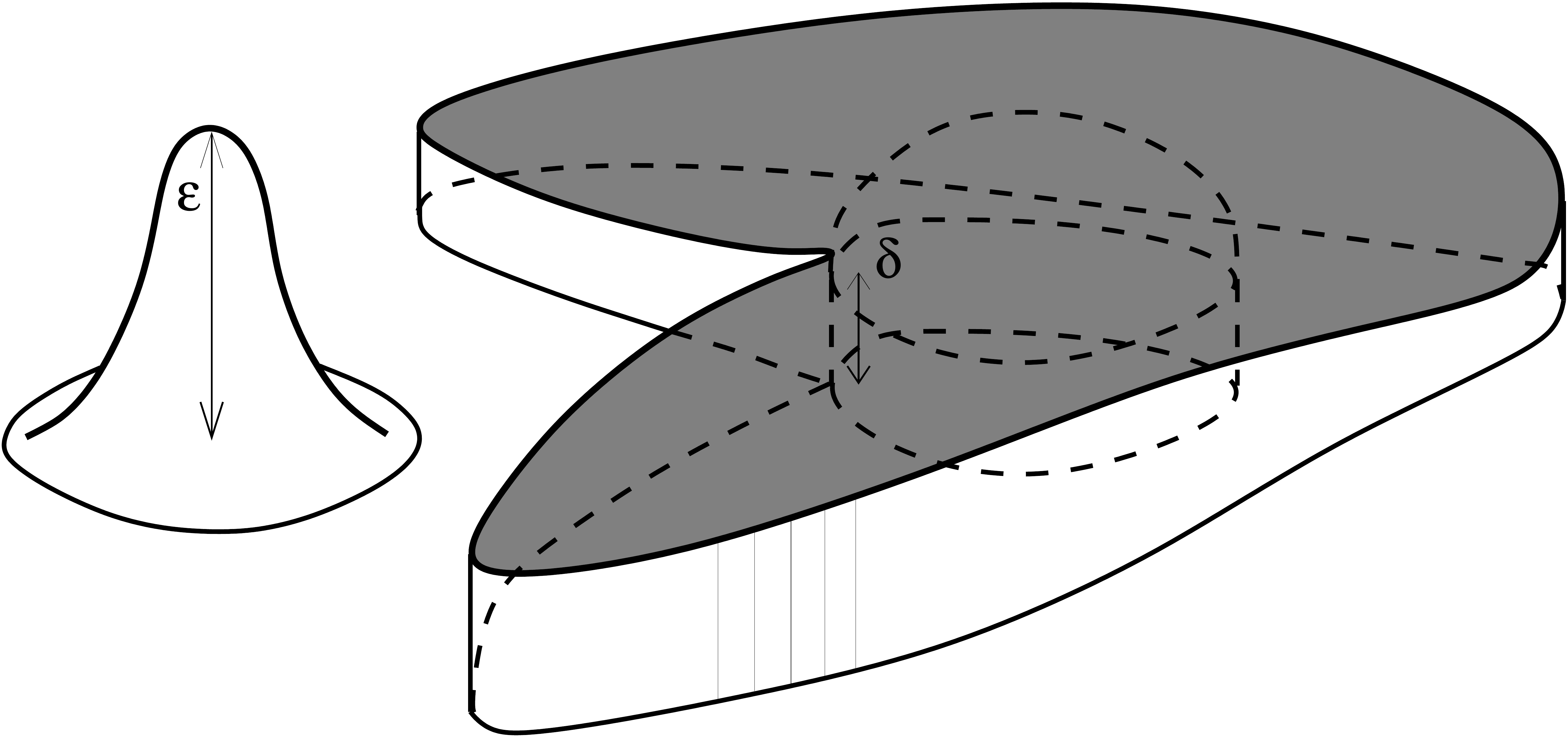}}
\caption{Boundary discontinuity for ``inward pointing'' domains. A perturbation of size~$\e$ producing a discontinuity of size~$\delta$.}\label{FIG2}
\end{figure}

The result corresponding to Theorem~\ref{AKgSMX:CONCA} is sketched in Figure~\ref{FIG2}.
\medskip

The rest of this paper is organized as follows.
In Section~\ref{apjsonxdkcbPOSDmsvd}, we present the proof of
Theorem~\ref{THM:CON}, from which one also deduces Corollary~\ref{THM:CON-cor}.
The proof of Theorem~\ref{AKgSMX:CONCA} is contained in Section~\ref{S2-30eM}.

\section{Proofs of Theorem~\ref{THM:CON} and Corollary~\ref{THM:CON-cor}}\label{apjsonxdkcbPOSDmsvd}

We start this section by noticing that $s$-minimal sets satisfy a suitable geometric equation
(in a suitable ``viscosity sense'') at domain boundary points.

\begin{lemma}\label{ATTSOL}
Let~$E$ be an $s$-minimal set in~$\Omega$.
Assume that~$P\in\partial E\cap\overline\Omega$.

Assume also that there exists an~$(n+1)$-dimensional ball~${\mathcal{B}}$ such that~${\mathcal{B}}\subseteq E^c$ and~$P\in\partial {\mathcal{B}}$.

Then,
\begin{equation}\label{FMCD}\int_{\R^{n+1}}\frac{\chi_{E^c}(X)-\chi_E(X)}{|X-P|^{n+1+s}}\,dX\le0,\end{equation}
where the integral\footnote{The integral on the left-hand side of~\eqref{FMCD}
is called in jargon the $s$-mean curvature of~$E$ at the point~$P$, \label{TSGDqw}
and it is often denoted by~${\mathcal{H}}^s_E(P)$.} is taken in the Cauchy Principal Value sense.
\end{lemma}

\begin{proof} Let~$A \subseteq E\cap \Omega$. Then, by minimality,
we have that
\begin{eqnarray*}
0\ge \Per_s(E,\Omega)-\Per_s(E\setminus A,\Omega)
= L_s(A,E^c)-L_s(A,E\setminus A).
\end{eqnarray*}
{F}rom this and~\cite[Theorem~5.1]{MR2675483} we obtain the desired result.
\end{proof}

Now we dive into the proof of Theorem~\ref{THM:CON}. The gist is that a discontinuity would produce
an $s$-minimal set with regularity no better than~$C^{1,s}$, thus producing an infinite $s$-mean curvature, in contradiction with the fact that $s$-minimal sets have vanishing $s$-mean curvature.
However, some care is needed to make such an argument rigorous, since the verification of the $s$-mean curvature equation at domain boundary points is a delicate issue (even more when the domain~$\Omega$ is not regular). The details go as follows:

\begin{proof}[Proof of Theorem~\ref{THM:CON}] Without loss of generality, in~\eqref{CBET} we can suppose that~$c\in(0,1)$.
Up to taking~$\rho$ smaller, we may also suppose that~$
\beta\ge1$. Therefore, if~$\mu\left(0,\frac18\right]$ and~$|y'|<\frac\mu2$,
\begin{equation}\label{XCKMx-10}
c|y'|^\beta \le |y'|^\beta\le|y'|<\frac{\mu}2
.\end{equation}

Now, to prove~\eqref{CONT:1} we argue by contradiction and suppose that
there exist~$\e_0>0$, an infinitesimal sequence~$\delta_k>0$ and points~$P_k=(p_k, p_{k,n+1})=(p_k',p_{k,n},p_{k,n+1})\in E$ with~$|p_k|<\delta_k$ and~$p_{k,n+1}>\e_0$. 
Actually, we can take~$P_k$ with the largest possible~$(n+1)$th coordinate,
namely renaming~$P_k=(p_k, p_{k,n+1})$ with~$p_{k,n+1}^\star=\sup\{t\in(\e_0,+\infty) {\mbox{ s.t. }} (p_k,t)\in E\}$. In this way, we can assume
that~$P_k\in\partial E$.

We take~$k$ sufficiently large such that~$\delta_k<\rho$. Hence, by~\eqref{INC:2}, we see that~$\psi(p'_k,p_{k,n})\le0$.
This and~\eqref{INC:22} entail that~$P_k\not\in E_0$ and therefore necessarily~$P_k\in\Omega$,
that is
\begin{equation}\label{Pkinom}p_k\in\omega.\end{equation}
Thus, from~\cite[Lemma~3.3]{MR3516886} we deduce that~$p_{k,n+1}$ is a bounded sequence.
Therefore, up to a subsequence, we can suppose that~$P_k\to P$ as~$k\to+\infty$,
for some~$P=(0,\dots,0,P_{n+1})\in\R^{n+1}$ with~$P_{n+1}\ge\e_0$.

We stress that the origin of~$\R^n$ belongs to~$\partial\omega$, due to~\eqref{KASMFMA}, therefore~$P\in\partial\Omega$.

We define
$$ \mu:=\frac{\min\left\{\e_0, \rho,1\right\}}8$$
and we claim that
\begin{equation}\label{CEM.a}
{\mathcal{B}}_\mu(0,\dots,0,-\mu,P_{n+1})\subseteq E^c .
\end{equation}
Indeed, if there were~$Q=(q',q_n,q_{n+1})\in E$ with~$|q'|^2+|q_n+\mu|^2+|q_{n+1}-P_{n+1}|^2<\mu^2$ then~$|q'|<\mu$ and~$|q_n+\mu|<\mu$. Therefore, $q_n\le|q_n+\mu|-\mu<0\le\varphi(q')$.
We also note that~$|(q',q_n)|\le|q'|+|q_n|\le3\mu<\rho$.

{F}rom these considerations and~\eqref{INC:1} we conclude that~$(q',q_n)\not\in\omega$, whence~$Q\not\in\Omega$.
Consequently, by~\eqref{INC:22} and~\eqref{INC:2}, we see that~$q_{n+1}\le\psi(q',q_n)\le0$.
For this reason,
$$ 8\mu\le\e_0+0\le P_{n+1}-q_{n+1}\le
|q_{n+1}-P_{n+1}|<\mu.
$$
This yields a contradiction, proving~\eqref{CEM.a}.

We also note that~$P\in\partial E$ and that~$P$
belongs to the boundary of the ball in~\eqref{CEM.a}.
Hence, by~\eqref{CEM.a} and Lemma~\ref{ATTSOL},
\begin{equation}\label{MAIJil2}
\int_{\R^{n+1}}\frac{\chi_{E^c}(X)-\chi_E(X)}{|X-P|^{n+1+s}}\,dX\le0.\end{equation}

Now we define
$$ {\mathcal{K}}_\mu:=
B_\mu\times (P_{n+1}-\mu,P_{n+1}+\mu)$$ and we
observe that
\begin{equation}\label{MAIJil}
{\mathcal{K}}_\mu\cap\{x_n\le\varphi(x')\}\,\subseteq\, E^c.
\end{equation}
Indeed, if~$\xi$ belongs to the set on the left-hand side of~\eqref{MAIJil}, we deduce from~\eqref{INC:1} that~$\xi\in\Omega^c$ and therefore the claim reduces to checking that
\begin{equation}\label{INC:1CT}
\xi\in E_0^c.\end{equation}
Since, by~\eqref{INC:2},
$$ \xi_{n+1}\ge P_{n+1}-\mu\ge\e_0 -\frac{\e_0}8>0\ge\psi(\xi',\xi_n),
$$
the claim in~\eqref{INC:1CT} follows from~\eqref{INC:22}. The proof of~\eqref{MAIJil} is thereby complete.

{F}rom~\eqref{MAIJil2} and~\eqref{MAIJil} we infer that
\begin{equation} \label{SAH-01}\begin{split}
0&\ge
\int_{{\mathcal{K}}_\mu}\frac{\chi_{E^c}(X)-\chi_E(X)}{|X-P|^{n+1+s}}\,dX
+
\int_{\R^{n+1}\setminus {\mathcal{K}}_\mu}\frac{\chi_{E^c}(X)-\chi_E(X)}{|X-P|^{n+1+s}}\,dX
\\&\ge
\int_{{\mathcal{K}}_\mu\cap\{x_n\le\varphi(x')\}}\frac{dX}{|X-P|^{n+1+s}}-
\int_{{\mathcal{K}}_\mu\cap\{x_n>\varphi(x')\}}\frac{dX}{|X-P|^{n+1+s}}
-\int_{\R^{n+1}\setminus {\mathcal{K}}_\mu}\frac{dX}{|X-P|^{n+1+s}}.
\end{split}\end{equation}
We also remark that
\begin{equation} \label{SAH-02} \int_{\R^{n+1}\setminus {\mathcal{K}}_\mu}\frac{dX}{|X-P|^{n+1+s}}\le\frac{C}{\mu^{s}},\end{equation}
for some~$C>0$ depending only on~$n$ and~$s$.

Besides, substituting for~$(y',y_n,y_{n+1}):=
(x',-x_n,x_{n+1})$, and noticing that~$P_n=0$,
\begin{eqnarray*}
\int_{{\mathcal{K}}_\mu\cap\{x_n\le\varphi(x')\}}\frac{dX}{|X-P|^{n+1+s}}=
\int_{{\mathcal{K}}_\mu\cap\{y_n\ge-\varphi(y')\}}\frac{dY}{|Y-P|^{n+1+s}},
\end{eqnarray*}
giving that
\begin{eqnarray*}&&
\int_{{\mathcal{K}}_\mu\cap\{x_n\le\varphi(x')\}}\frac{dX}{|X-P|^{n+1+s}}-
\int_{{\mathcal{K}}_\mu\cap\{x_n>\varphi(x')\}}\frac{dX}{|X-P|^{n+1+s}}\\&=&
\int_{{\mathcal{K}}_\mu\cap\{x_n\ge-\varphi(x')\}}\frac{dX}{|X-P|^{n+1+s}}
-
\int_{{\mathcal{K}}_\mu\cap\{x_n>\varphi(x')\}}\frac{dX}{|X-P|^{n+1+s}}\\&=&
\int_{{\mathcal{K}}_\mu\cap\{|x_n|<\varphi(x')\}}\frac{dX}{|X-P|^{n+1+s}}.
\end{eqnarray*}

We thereby combine this information with~\eqref{SAH-01} and~\eqref{SAH-02} to conclude that
\begin{equation}\label{PAJOSLXN-0i2u3r9iyhg}
\frac{C}{\mu^{s}}\ge\int_{{\mathcal{K}}_\mu\cap\{|x_n|<\varphi(x')\}}\frac{dX}{|X-P|^{n+1+s}}.
\end{equation}

Now we use the short notation~$\mu_k:=\frac\mu{2^k}$.
By~\eqref{CBET} and~\eqref{XCKMx-10}, up to renaming~$C$ line after line, we have that
\begin{eqnarray*}&&
\int_{{\mathcal{K}}_\mu\cap\{|x_n|<\varphi(x')\}}\frac{dX}{|X-P|^{n+1+s}}\\&\ge&
\int_{\{|x'|<\mu/2\}\times\{|x_n|<\min\{c|x'|^\beta,\mu/2\}\}\times\{|x_{n+1}-P_{n+1}|<\mu\}}\frac{dX}{|X-P|^{n+1+s}}\\&=&\int_{\{|y'|<\mu/2\}\times\{|y_n|<\min\{c|y'|^\beta,\mu/2\}\}\times\{|y_{n+1}|<\mu\}}\frac{dY}{|Y|^{n+1+s}}
\\&\ge&
\int_{\{|y'|<\mu/2\}\times\{|y_n|< c|y'|^\beta\}\times\{|y_{n+1}|<\mu\}}\frac{dY}{|Y|^{n+1+s}}
\\&\ge&\sum_{k=1}^{+\infty}\int_{\{\mu/2^{k+1}<|y'|<\mu/2^k\}\times\{c|y'|^\beta/2<|y_n|<c|y'|^\beta\}\times\{\mu/2^{k+1}<|y_{n+1}|<\mu/2^k\}}\frac{dY}{|Y|^{n+1+s}}
\\&\ge&\sum_{k=1}^{+\infty}\frac1C\int_{\{\mu_k/2<|y'|<\mu_k\}\times\{c|y'|^\beta/2<|y_n|<c|y'|^\beta\}\times\{\mu_k/2<|y_{n+1}|<\mu_k\}}\frac{dY}{\big( \mu_k^2+c^2|y'|^{2\beta}\big)^{\frac{n+1+s}2}}
\\&\ge&\sum_{k=1}^{+\infty}\frac{c\mu_k}C\int_{\{\mu_k/2<|y'|<\mu_k\} }\frac{|y'|^\beta\,dy'}{\big( \mu_k^2+c^2|y'|^{2\beta}\big)^{\frac{n+1+s}2}}\\&\ge&\sum_{k=1}^{+\infty}
\frac{c\mu_k^{\beta+n}}{C\big( \mu_k^2+c^2\mu^{2\beta}_k\big)^{\frac{n+1+s}2}}\\&=&\sum_{k=1}^{+\infty}
\frac{c\mu_k^{\beta-1-s}}{C\big( 1+c^2\mu_k^{2(\beta-1)}\big)^{\frac{n+1+s}2}}\\
\\&\ge&\sum_{k=1}^{+\infty}
\frac{c\mu^{\beta-1-s} \;2^{k(1+s-\beta)}}{C(1+c^2)^{\frac{n+1+s}2}}.
\end{eqnarray*}
The latter is a divergent series, since~$\beta\le s+1$. But this is in contradiction with~\eqref{PAJOSLXN-0i2u3r9iyhg} and, as a result of this, the proof of Theorem~\ref{THM:CON} is complete.\end{proof}

\begin{proof}[Proof of Corollary~\ref{THM:CON-cor}] Up to a rotation, we can describe
the convex corner at the origin by writing~$\omega$
in the form~$x_n>c|x'|$, for some~$c>0$. This gives that the setting
in~\eqref{KASMFMA}, \eqref{CBET},
and~\eqref{INC:1} is satisfied, with~$n:=2$ and~$\beta:=1$.

We are therefore in the framework of~\eqref{BSODIKnC1}, whence the desired result follows from~\eqref{BSODIKnC2}.
\end{proof}

\section{Proof of Theorem~\ref{AKgSMX:CONCA}}\label{S2-30eM}

The argument presented here will rely on a convenient barrier. To construct it, we start with some preliminary computations.

\begin{lemma}\label{LDE-31}
Consider a bounded, $n$-dimensional set~$S$
with boundary of class~$C^{1,\alpha}$, for some~$\alpha\in(0,1)$, with~$0\in\partial S$.

There exists~$w\in C^{1,\alpha}(\R^n)$ such that~$w\ge0$ in~$S$, $w\le0$ in~$\R^n\setminus S$, and
\begin{equation}\label{IFDSP} \liminf_{ x\to 0}|\nabla w(x)|\ge1.\end{equation}
\end{lemma}

\begin{proof} Up to a rotation, we can assume that~$S\cap B_\rho=\{x_n>\phi(x')\}\cap B_\rho$ for some~$\rho>0$ and~$\phi\in C^{1,\alpha}(\R^{n-1})$. Let~$\tau\in C^\infty_0(B_\rho,[0,1])$ with~$\tau=1$ in~$B_{\rho/2}$ and
$$ w(x):=\big( x_n-\phi(x')\big)\, \tau(x).$$
Notice that
$$ \nabla w(x)=\big( -\nabla_{x'}\phi(x'),1\big)\, \tau(x)+\big( x_n-\phi(x')\big)\,\nabla \tau(x),$$
and therefore
$$ \lim_{ x\to 0}\nabla w(x)=\big( -\nabla_{x'}\phi(0),1\big),$$
from which the desired result follows.
\end{proof}

\begin{lemma}\label{LDE-32}
Let~$\e>0$ and~$\alpha\in(s,1)$. Let also~$S$ be an open subset of~$\R^n$.

Let~$w\in C^{1,\alpha}(\R^n)$ such that~$w\ge0$ in~$S$ and~$w\le0$ in~$\R^n\setminus S$.
Let also~$w_+(x):=\max\{w(x),0\}$ and~$W:=\{x_{n+1}<\e w_+(x)\}$.

Then, the $s$-mean curvature of~$W$ at every point on~$(\partial W)\cap (S\times\R)$ is bounded from above
by
$$C\big(\|w\|_{C^{1,\alpha}(\R^n)}\,\e\big)^{\frac{s}\alpha},$$
where~$C>0$ depends only on~$n$ and~$s$.
\end{lemma}

\begin{proof} Let~$X=(x,x_{n+1})\in(\partial W)\cap (S\times\R)$. Then, $x\in S$ and~$x_{n+1}=\e w_+(x)=\e w(x)$.
In this way, by~\cite[equation~(49)]{MR3331523}, up to normalizing constant, the $s$-mean curvature of~$W$ at~$X$ is equal to
\begin{equation}\label{OTDG-00}
\int_{\R^n} F\left(\frac{\e\big(w_+(x)-w_+(x-y)\big)}{|y|}\right)\,\frac{dy}{|y|^{n+s}}=\int_{\R^n} F\left(\frac{\e\big(w(x)-w_+(x-y)\big)}{|y|}\right)\,\frac{dy}{|y|^{n+s}},\end{equation}
where
$$ F(t):=\int_0^t\frac{d\tau}{(1+\tau^2)^{\frac{n+1+s}2}}.$$

Since~$F$ is monotone and~$w_+(x-y)\ge w(x-y)$, we have that
$$ F\left(\frac{\e\big(w(x)-w_+(x-y)\big)}{|y|}\right)\le
F\left(\frac{\e\big(w(x)-w(x-y)\big)}{|y|}\right)$$
and accordingly the quantity in~\eqref{OTDG-00} is bounded from above by
\begin{equation}\label{OTDG-0}\int_{\R^n} F\left(\frac{\e\big(w(x)-w(x-y)\big)}{|y|}\right)\,\frac{dy}{|y|^{n+s}}.\end{equation}

We also remark that~$F$ is odd and therefore
$$ \int_{\R^n} F\left(\frac{\e\nabla w(x)\cdot y}{|y|}\right)\,\frac{dy}{|y|^{n+s}}=0,$$
hence we can rewrite~\eqref{OTDG-0} in the form
\begin{equation}\label{OTDG-1}
\int_{\R^n} \left[F\left(\frac{\e\big(w(x)-w(x-y)\big)}{|y|}\right)-F\left(\frac{\e\nabla w(x)\cdot y}{|y|}\right)\right]\,\frac{dy}{|y|^{n+s}}.\end{equation}

Now we observe that
\begin{eqnarray*}&&F\left(\frac{\e\big(w(x)-w(x-y)\big)}{|y|}\right)-F\left(\frac{\e\nabla w(x)\cdot y}{|y|}\right)\\&=&
\int_0^1 F'\left(\frac{\e}{|y|}\Big((1-t)\big(w(x)-w(x-y)\big)+t\nabla w(x)\Big)\right)\,dt\;
\frac{\e}{|y|}\Big( w(x)-w(x-y)-\nabla w(x)\cdot y\Big)\\& \le& \frac{C\e}{|y|}\Big| w(x)-w(x-y)-\nabla w(x)\cdot y\Big|\\& \le& \frac{C\e}{|y|}\left| \int_0^1 \nabla w(x-\theta y)\cdot y\,d\theta-\nabla w(x)\cdot y\right|\\& \le& C\e \int_0^1 \big|\nabla w(x-\theta y)-\nabla w(x)\big|\,d\theta\\&\le& C\|w\|_{C^{1,\alpha}(\R^n)}\,\e\, |y|^\alpha
\end{eqnarray*}
and therefore, since~$\alpha>s$, for all~$R>0$,
\begin{equation}\label{olw-pdijfec}
\int_{B_R} \left[F\left(\frac{\e\big(w(x)-w(x-y)\big)}{|y|}\right)-F\left(\frac{\e\nabla w(x)\cdot y}{|y|}\right)\right]\,\frac{dy}{|y|^{n+s}}\le C\|w\|_{C^{1,\alpha}(\R^n)}\,\e R^{\alpha-s} .\end{equation}

Besides, since~$F$ is bounded,
$$ \int_{\R^n\setminus B_R} \left[F\left(\frac{\e\big(w(x)-w(x-y)\big)}{|y|}\right)-F\left(\frac{\e\nabla w(x)\cdot y}{|y|}\right)\right]\,\frac{dy}{|y|^{n+s}}\le\frac{C}{R^s}.$$

This and~\eqref{olw-pdijfec} give that the quantity in~\eqref{OTDG-1} is bounded from above by~$C\|w\|_{C^{1,\alpha}(\R^n)}\,\e R^{\alpha-s}+\frac{C}{R^s}$.

It is now convenient to choose
$$R:=\frac1{\big(\|w\|_{C^{1,\alpha}(\R^n)}\,\e\big)^{\frac1\alpha}}$$
to obtain the desired result.
\end{proof}

For us, Lemmata~\ref{LDE-31} and~\ref{LDE-32} come in handy to construct a useful barrier, see Figure~\ref{FIGAGG7}:

\begin{figure}[h]
\fbox{\includegraphics[height=4.5cm]{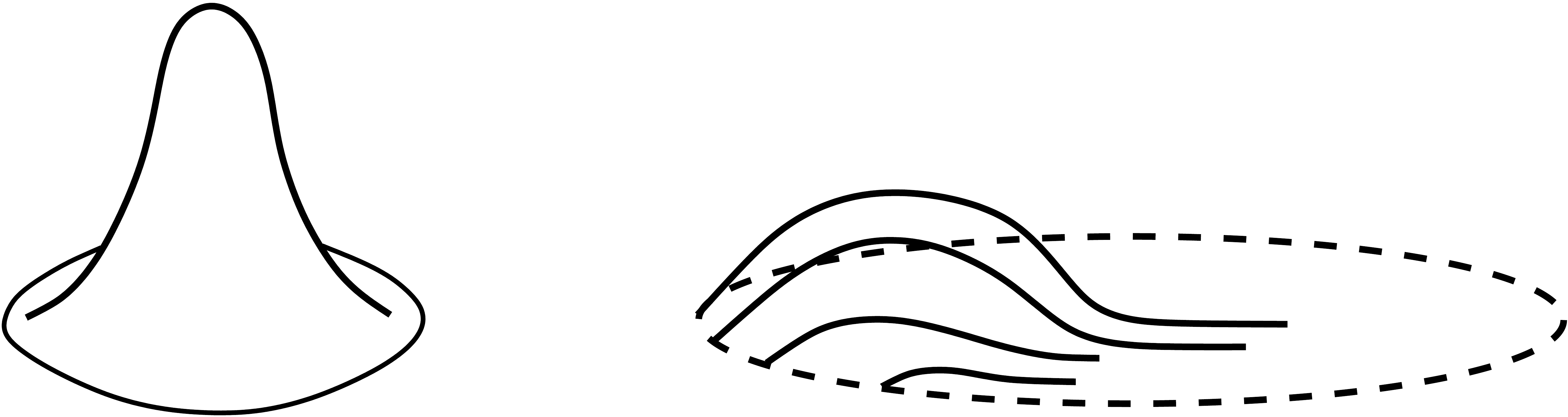}}
\caption{The barrier $v$ in Corollary~\ref{CORBA}.}\label{FIGAGG7}
\end{figure}

\begin{corollary}\label{CORBA}
Consider a bounded, $n$-dimensional set~$S$
with boundary of class~$C^{1,\alpha}$, for some~$\alpha\in(s,1)$, with~$0\in\partial S$.

Let also~$\omega$ be a bounded, open set in~$\R^n$ such that~$\omega\supseteq S$.

Then, for any~$\e>0$ small enough, there exists~$v\in C^{0,1}_0(\R^n,[0,\e])$ with~$v=0$ in~$\varpi\setminus\omega$ for some open set~$\varpi$ with~$\omega\Subset\varpi$,
such that, setting~$V:=\{x_{n+1}<v(x)\}$, we have that the $s$-mean curvature of~$(\partial V)\cap(S\times\R)$ is strictly negative
and
\begin{equation}\label{K12SM3D-w2ep3roefp0}  \liminf_{ S\ni x\to 0}|\nabla v(x)|\ne0.\end{equation}
\end{corollary}

\begin{proof} We pick a ball~$B_1(p_0)$ in~$\R^n$ such that~$\overline{B_1(p_0)}\cap\overline\omega=\varnothing$.
We take~$\tau\in C^\infty_0(B_1(p_0),[0,1])$ such that~$\tau=1$ in~$B_{1/2}(p_0)$.
Let also~$w$ be as in Lemma~\ref{LDE-31} and correspondingly we let~$W$ be as in Lemma~\ref{LDE-32}.

We
define
$$ v:=\e^\gamma\tau+ \e w_+ ,$$
with~$\gamma\in\left(0,\frac{s}\alpha\right)$.

We remark that~$v\ge\e w_+$ and consequently~$V\supseteq W$, giving that~$V^c\setminus W^c=\varnothing$.

Moreover, since~$w_+$ and~$\tau$ have disjoint supports,
\begin{eqnarray*}&& V\setminus W =\big\{ v(x)> x_{n+1}\ge \e w_+(x)\big\}=
\big\{ \e^\gamma\tau(x)+ \e w_+(x)> x_{n+1}\ge\e w_+(x)\big\}
\\&&\qquad=
\big\{ \e^\gamma\tau(x)> x_{n+1}-\e w_+(x)\ge0\big\}
\supseteq B_{1/2}(p_0)\times [0,\e^\gamma).
\end{eqnarray*}
For that reason, using the notation in footnote~\ref{TSGDqw}, if~$P\in(\partial V)\cap(S\times\R)$, owing to Lemma~\ref{LDE-32}, we find that
\begin{eqnarray*}
{\mathcal{H}}^s_V(P)&=&{\mathcal{H}}^s_W(P)+
\int_{\R^n}\frac{\chi_{V^c\setminus W^c}(X)-\chi_{V\setminus W}(X)}{|X-P|^{n+1+s}}\,dX\\&\le& C\big(\|w\|_{C^{1,\alpha}(\R^n)}\,\e\big)^{\frac{s}\alpha}-\int_{B_{1/2}(p_0)\times [0,\e^\gamma)}\frac{dX}{|X-P|^{n+1+s}}\\&\le& C\big(\|w\|_{C^{1,\alpha}(\R^n)}\,\e\big)^{\frac{s}\alpha}-c\e^\gamma.
\end{eqnarray*}
This quantity is negative when~$\e$ is small enough, as desired.

In addition, by~\eqref{IFDSP} and the fact that~$w\ge0$ in~$S$,
$$ \liminf_{ S\ni x\to 0}|\nabla v(x)|=\e\liminf_{S\ni x\to 0}|\nabla w_+(x)|=\e\liminf_{S\ni x\to 0}|\nabla w(x)|\ge\e\ne0.$$

Notice also that~$v$ is bounded by a power of~$\e$, hence the desired claim follows up to renaming~$\e$.
\end{proof}

We are now in the position of completing the proof of  Theorem~\ref{AKgSMX:CONCA} by combining
the barrier constructed in Corollary~\ref{CORBA} and a careful blow-up method\footnote{For simplicity of notation, some of the arguments were presented in~\cite{MR4104542} for the two-dimensional case, but, as remarked at the beginning of Section~2 there, the $n$-dimensional analysis
would have remained unaltered.} introduced in~\cite{MR4104542}.

\begin{proof}[Proof of Theorem~\ref{AKgSMX:CONCA}] Up to a rotation, we assume that, near the origin, the inward touching domain~$S$ is the superlevel set of a function of class~$C^{1,\alpha}$ in the $n$th Cartesian coordinate.

If~$v$ and~$V$ are as in Corollary~\ref{CORBA}, we can take~$\psi\in C^\infty_0(\R^n,[0,\e])$ such that~$\psi\ge v$ in~$\R^n\setminus\omega$ and~$\psi=0$ in~$\omega$.

By the comparison principle in~\cite[Section~5]{MR2675483}, we have that
\begin{equation}\label{K12SM3D-w2ep3roef}
E\supseteq V.\end{equation}

Now we perform a blow-up at the origin, using~\cite[Lemmata~2.2 and~2.3]{MR4104542},
obtaining in this way an $s$-minimal cone~$E_{00}$ in (up to a rotation) $\{x_n>0\}$.
By~\eqref{K12SM3D-w2ep3roefp0} and~\eqref{K12SM3D-w2ep3roef}, it follows that~$E_{00}$ presents a corner at the origin
with a nontrivial slope in the vertical direction.

By~\cite[Theorem 4.1]{MR4104542}, we obtain that there exists~$\delta>0$ such that either
\begin{equation}\label{EUIJSNTH-1}
E\cap {\mathcal{B}}_\delta\cap(S\times\R)=\varnothing.
\end{equation}
or
\begin{equation}\label{EUIJSNTH-2}
E\cap {\mathcal{B}}_\delta\cap(S\times\R)={\mathcal{B}}_\delta\cap(S\times\R).
\end{equation}
But~\eqref{EUIJSNTH-1} cannot hold true, in light of~\eqref{K12SM3D-w2ep3roef}, hence necessarily~\eqref{EUIJSNTH-2} is satisfied, which gives the desired result in~\eqref{siCONT:1}, up to renaming~$\delta$.
\end{proof}

\end{document}